\documentclass[twoside,reqno]{amsart}

\usepackage{latexsym}

\newcommand{\qdn}{\hspace*{-1.5mm}}
\newcommand{\qqdn}{\hspace*{-2.5mm}}
\newcommand{\xqdn}{\hspace*{-5.0mm}}
\newcommand{\xxqdn}{\hspace*{-10mm}}

\newcommand{\ffnk}[4]{\left[\qdn\ba{#1}#3\\[2mm]#4\ea{\!\bigg|\:#2}\right]}

\newcommand{\binm}{\binom}

\newcommand{\nnm}{\nonumber}
\newcommand{\be}{\begin{equation}}
\newcommand{\ee}{\end{equation}}
\newcommand{\ba}{\begin{array}}
\newcommand{\ea}{\end{array}}
\newcommand{\bmn}{\begin{eqnarray}}
\newcommand{\emn}{\end{eqnarray}}
\newcommand{\bnm}{\begin{eqnarray*}}
\newcommand{\enm}{\end{eqnarray*}}
\newcommand{\bln}{\begin{subequations}}
\newcommand{\eln}{\end{subequations}}

\newtheorem{thm}{Theorem}

\newtheorem{entry}{Entry}

\newcommand{\bbtm}[4]{\bibitem{kn:#1}{#2,}~{#3,}~{#4.}}
\newcommand{\cito}[1]{\cite{kn:#1}}
\newcommand{\citu}[2]{\cite[#2]{kn:#1}}

\begin{document} 

\title{A symmetric formula for hypergeometric series}
\author{Chuanan Wei}


\dedicatory{Department of Medical Informatics\\
 Hainan Medical University, Haikou 571199, China}
\thanks{\emph{Email address}: weichuanan78@163.com}


\begin{abstract}
In terms of Dougall's $_2H_2$ series identity and the series
rearrangement method, we establish an interesting symmetric formula
for hypergeometric series. Then it is utilized to derive a known
nonterminating form of Saalsch\"{u}tz's theorem. Similarly, we also
show that Bailey's $_6\psi_6$ series identity implies the
nonterminating form of Jackson's $_8\phi_7$ summation formula.
Considering the reversibility of the proofs, it is routine to show
that Dougall's $_2H_2$ series identity is equivalent to a known
nonterminating form of Saalsch\"{u}tz's theorem and Bailey's
$_6\psi_6$ series identity is equivalent to the nonterminating form
of Jackson's $_8\phi_7$ summation formula.
\end{abstract}

\maketitle\thispagestyle{empty}
\markboth{Chuanan Wei}
         {A symmetric formula for hypergeometric series}
\section{Introduction}
For an integer $n$ and a complex number $x$, define the shifted
factorial to be
 \[(x)_n=\Gamma(x+n)/\Gamma(x),\]
where $\Gamma(x)$ is the well known gamma function
\[\Gamma(x)=\int_{0}^{\infty}t^{x-1}e^{-t}dt\quad\text{with}\quad Re(x)>0.\]
 Following Andrews, Askey and Roy ~\cito{andrews-b}, define the hypergeometric series
  by
\[\qqdn\xqdn_{1+r}F_s\ffnk{cccc}{z}{a_0,&a_1,&\cdots,&a_r}{&b_1,&\cdots,&b_s}
 \:=\:\sum_{k=0}^\infty
\frac{(a_{0})_{k}(a_{1})_{k}\cdots(a_{r})_{k}}
 {k!(b_{1})_{k}\cdots(b_{s})_{k}}z^k.\]
Then Saalsch\"{u}tz's theorem (cf. \citu{andrews-b}{p. 69}) can be
stated as
 \bmn\label{saal}
 \qqdn\xxqdn_3F_2\ffnk{cccc}{1}{a,b,-n}{c,1+a+b-c-n}
=\frac{(c-a)_n(c-b)_n} {(c)_n(c-a-b)_n}.
 \emn
 A known nonterminating form of it (cf. \citu{andrews-b}{p. 92}) reads as
 \bmn\label{saal-non}
&&\xxqdn{_3F_2}\ffnk{ccc}{1}{a,b,c+d-a-b-1}{c,d}=
{_3F_2}\ffnk{ccc}{1}{1,c-a,c-b}{c-a-b+1,c+d-a-b}
 \nnm\\[1mm]&&\xxqdn\:\:\times\:\:\,
 \frac{\Gamma(c)\Gamma(d)}{\Gamma(a)\Gamma(b)\Gamma(c+d-a-b)}\frac{1}{a+b-c}
 +\frac{\Gamma(c)\Gamma(d)\Gamma(c-a-b)\Gamma(d-a-b)}{\Gamma(c-a)\Gamma(c-b)\Gamma(d-a)\Gamma(d-b)},
 \emn
provided $Re(d-a-b)>0$. The known proof of \eqref{saal-non} comes
from a transformation formula involving three $_3F_2$ series given
by the contour integration method. The reader is referred to
\citu{andrews-b}{Section 2.4} for details.

 Following Slater \cito{slater}, define the bilateral
hypergeometric series to be
 \bnm
\:\xqdn_{r}H_s\ffnk{cccc}{z}{a_1,&a_2,&\cdots,&a_r}{b_1,&b_2,&\cdots,&b_s}
 =\sum_{k=-\infty}^\infty
\frac{(a_1)_k(a_2)\cdots(a_r)_k}{(b_1)_k(b_2)_k\cdots(b_s)_k}z^k.
 \enm
Thus Dougall's $_2H_2$ series identity (cf. \citu{andrews-b}{p.
110}) can be written as
 \bmn\label{h22}
 {_2H_2}\ffnk{ccc}{1}{a,b}{c,d}
 =\frac{\Gamma(1-a)\Gamma(1-b)\Gamma(c)\Gamma(d)\Gamma(c+d-a-b-1)}{\Gamma(c-a)\Gamma(c-b)\Gamma(d-a)\Gamma(d-b)},
 \emn
where $Re(c+d-a-b)>1$.

For an integer $n$ and two complex numbers $x$, $q$ with $|q|<1$,
define the $q$-shifted factorial by
 \bnm
(x;q)_{\infty}=\prod_{i=0}^{\infty}(1-xq^i),\quad
(x;q)_n=\frac{(x;q)_{\infty}}{(xq^n;q)_{\infty}}.
 \enm
For simplification, we shall frequently adopt the following
notations:
 \bnm
&&(x_1,x_2,\cdots,x_r;q)_{\infty}=(x_1;q)_{\infty}(x_2;q)_{\infty}\cdots(x_r;q)_{\infty},\\
&&(x_1,x_2,\cdots,x_r;q)_{n}=(x_1;q)_{n}(x_2;q)_{n}\cdots(x_r;q)_{n}.
 \enm
Following Gasper and Rahman \cito{gasper}, define the basic
hypergeometric series and bilateral basic hypergeometric series to
be
 \bnm
&&\xqdn{_{1+r}\phi_s}\ffnk{cccccc}{q;z}{a_0,&a_1,&\cdots,a_r}{&b_1,&\cdots,b_s}
  =\sum_{k=0}^{\infty}\frac{(a_0,a_1,\cdots,a_r;q)_k}{(q,b_1,\cdots,b_s;q)_k}
\Big\{(-1)^kq^{\binm{k}{2}}\Big\}^{s-r}z^k,\\
&&\qdn{_r\psi_s}\ffnk{cccccc}{q;z}{a_1,&a_2,&\cdots,a_r}{b_1,&b_2,&\cdots,b_s}
  =\sum_{k=-\infty}^{\infty}\frac{(a_1,a_2,\cdots,a_r;q)_k}{(b_1,b_2,\cdots,b_s;q)_k}
\Big\{(-1)^kq^{\binm{k}{2}}\Big\}^{s-r}z^k.
 \enm
Then the nonterminating form of Jackson's $_8\phi_7$ summation
formula (cf. \citu{gasper}{p.54}) and Bailey's $_6\psi_6$ series
identity (cf. \citu{gasper}{p.140}) can be expressed as
\bmn\label{jackon-non}
  &&{_8\phi_7}\ffnk{ccccccc}{q;q}{a,q\sqrt{a},-q\sqrt{a},b,c,d,e,f}{\sqrt{a},-\sqrt{a},qa/b,qa/c,qa/d,qa/e,qa/f}
 \nnm\\[1mm]\nnm&&=
\frac{b}{a}\frac{(qa,c,d,e,f,qb/a,qb/c,qb/d,qb/e,qb/f;q)_{\infty}}{(qa/b,qa/c,qa/d,qa/e,qa/f,bc/a,bd/a,be/a,bf/a,qb^2/a;q)_{\infty}}
 \\[1mm]&&\:\times\:
 {_8\phi_7}\ffnk{ccccccc}{q;q}{b^2/a,qb/\sqrt{a},-qb/\sqrt{a},b,bc/a,bd/a,be/a,bf/a}
  {b/\sqrt{a},-b/\sqrt{a},qb/a,qb/c,qb/d,qb/e,qb/f}
  \nnm\\[1mm]&&\:+\:
\frac{(qa,b/a,qa/cd,qa/ce,qa/cf,qa/de,qa/df,qa/ef;q)_{\infty}}{(qa/c,qa/d,qa/e,qa/f,bc/a,bd/a,be/a,bf/a;q)_{\infty}}
 \emn
with $qa^2=bcdef$,
  \bmn\label{bailey}
  &&\xqdn{_6\psi_6}\ffnk{cccccccccc}{q;\frac{qa^2}{bcde}}
 {q\sqrt{a},-q\sqrt{a},b,c,d,e}
 {\sqrt{a},-\sqrt{a},qa/b,qa/c,qa/d,qa/e}
 \nnm\\[1mm]&&\xqdn\:\,=\:
 \ffnk{ccccc}{q}{q,qa,q/a,qa/bc,qa/bd,qa/be,qa/cd,qa/ce,qa/de}
 {q/b,q/c,q/d,q/e,qa/b,qa/c,qa/d,qa/e,qa^2/bcde}_{\infty},
 \emn
provided $|qa^2/bcde|<1$. The original proof of \eqref{jackon-non}
comes from a three term relation of $_8\phi_7$ series offered by the
$q$-integration method. The reader may consult \citu{gasper}{Section
2.11} for details. Recently, the research of $q$-congruence attracts
several mathematicians. Some nice results can be seen in the papers
\cite{kn:guo-a,kn:guo-b}.

In 2006, Chen and Fu \cito{chen} established some semi-finite forms
of bilateral basic hypergeometric series in accordance with Cauchy's
method. Subsequently, Jouhet \cito{jouhet} deduced \eqref{bailey}
from \eqref{jackon-non} in the same way. Several years later, Wei,
Yan and Li \cito{wei-a} derived similarly \eqref{h22} from
\eqref{saal-non}. More results related to Cauchy's method can be
found in the papers \cite{kn:zhang-a,kn:zhang-b}.

Inspired by the works just mentioned, it is natural to consider the
inverse of Cauchy's method. According to the series rearrangement
method, we shall deduce \eqref{saal-non} from \eqref{h22} in Section
2 and show that \eqref{bailey} implies \eqref{jackon-non} in Section
3.

\section{A symmetric formula for hypergeometric series}
\begin{thm}\label{thm}
Let $a,b,c,d$ be complex numbers. Then
 \bnm
\Phi(a,b;c,d)+\Phi(c,d;a,b)=\frac{\Gamma(a)\Gamma(b)\Gamma(c)\Gamma(d)\Gamma(a+b+c+d-1)}
{\Gamma(a+c)\Gamma(a+d)\Gamma(b+c)\Gamma(b+d)},
 \enm
where the symbol on the left hand side stands for
 \bnm
\Phi(a,b;c,d)=\sum_{k=0}^{\infty}\frac{\Gamma(a+k)\Gamma(b+k)\Gamma(a+b+c+d-1+k)}
{\Gamma(1+k)\Gamma(a+b+c+k)\Gamma(a+b+d+k)}.
 \enm
\end{thm}

\begin{proof}
Split the bilateral series into two parts to obtain
 \bnm
&&{_2H_2}\ffnk{ccc}{1}{1-a,1-b}{1+c,1+d}\\[1mm]
&&\:\:=\:\sum_{i=0}^{\infty}\frac{(1-a)_i(1-b)_i}{(1+c)_i(1+d)_i}
+\sum_{i=-\infty}^{-1}\frac{(1-a)_i(1-b)_i}{(1+c)_i(1+d)_i}\\[1mm]
&&\:\:=\:\sum_{i=0}^{\infty}\frac{(1-a)_i(1-b)_i}{(1+c)_i(1+d)_i}
+\frac{cd}{ab}\sum_{i=0}^{\infty}\frac{(1-c)_i(1-d)_i}{(1+a)_i(1+b)_i}\\[1mm]
&&\:\:=\:\frac{\Gamma(1+c)\Gamma(1+d)}{\Gamma(1-a)\Gamma(1-b)}
\sum_{i=0}^{\infty}\frac{\Gamma(1-a+i)\Gamma(1-b+i)}{\Gamma(1+c+i)\Gamma(1+d+i)}\\[1mm]
&&\:\:+\:\:\frac{\Gamma(a)\Gamma(b)}{\Gamma(-c)\Gamma(-d)}
\sum_{i=0}^{\infty}\frac{\Gamma(1-c+i)\Gamma(1-d+i)}{\Gamma(1+a+i)\Gamma(1+b+i)}.
 \enm
By means of the case $d=1$ of \eqref{h22}:
 \bmn\label{f21}
 {_2F_1}\ffnk{ccc}{1}{a,\:b}{c}
 =\frac{\Gamma(c)\Gamma(c-a-b)}{\Gamma(c-a)\Gamma(c-b)},
 \emn
 we can proceed as follows:
 \bnm
&&{_2H_2}\ffnk{ccc}{1}{1-a,1-b}{1+c,1+d}\\[1mm]
&&\:\:=\:\frac{\Gamma(1+c)\Gamma(1+d)}{\Gamma(1-a)\Gamma(1-b)}
\sum_{i=0}^{\infty}\frac{\Gamma(1-a+i)\Gamma(1-b+i)}{\Gamma(1+i)\Gamma(1+c+d+i)}
 {_2F_1}\ffnk{ccc}{1}{c,d}{1+c+d+i}\\[1mm]
&&\:\:+\:\:\frac{\Gamma(a)\Gamma(b)}{\Gamma(-c)\Gamma(-d)}
\sum_{i=0}^{\infty}\frac{\Gamma(1-c+i)\Gamma(1-d+i)}{\Gamma(1+i)\Gamma(1+a+b+i)}
 {_2F_1}\ffnk{ccc}{1}{a,b}{1+a+b+i}
 \enm
 \bnm
&&\:\:=\:\frac{cd}{\Gamma(1-a)\Gamma(1-b)}
\sum_{i=0}^{\infty}\sum_{k=0}^{\infty}
\frac{\Gamma(1-a+i)\Gamma(1-b+i)\Gamma(c+k)\Gamma(d+k)}{\Gamma(1+i)\Gamma(1+c+d+i+k)\Gamma(1+k)}\\[1mm]
&&\:\:+\:\:\frac{1}{\Gamma(-c)\Gamma(-d)}
\sum_{i=0}^{\infty}\sum_{k=0}^{\infty}
\frac{\Gamma(1-c+i)\Gamma(1-d+i)\Gamma(a+k)\Gamma(b+k)}{\Gamma(1+i)\Gamma(1+a+b+i+k)\Gamma(1+k)}\\[1mm]
&&\:\:=\:cd\sum_{k=0}^{\infty}
\frac{\Gamma(c+k)\Gamma(d+k)}{\Gamma(1+k)\Gamma(1+c+d+k)}{_2F_1}\ffnk{ccc}{1}{1-a,1-b}{1+c+d+k}\\[1mm]
&&\:\:+\:\:cd\sum_{k=0}^{\infty}
\frac{\Gamma(a+k)\Gamma(b+k)}{\Gamma(1+k)\Gamma(1+a+b+k)}{_2F_1}\ffnk{ccc}{1}{1-c,1-d}{1+a+b+k}\\[1mm]
&&\:\:=\:cd\sum_{k=0}^{\infty}\frac{\Gamma(c+k)\Gamma(d+k)\Gamma(a+b+c+d-1+k)}
{\Gamma(1+k)\Gamma(a+c+d+k)\Gamma(b+c+d+k)}\\[1mm]
&&\:\:+\:\:cd\sum_{k=0}^{\infty}\frac{\Gamma(a+k)\Gamma(b+k)\Gamma(a+b+c+d-1+k)}
{\Gamma(1+k)\Gamma(a+b+c+k)\Gamma(a+b+d+k)}.
 \enm
Employing the substitutions $a\to1-a,b\to1-b,c\to1+c,d\to1+d$ in
\eqref{h22}, we get
 \bnm
{_2H_2}\ffnk{ccc}{1}{1-a,\:1-b}{1+c,\:1+d}=
\frac{\Gamma(a)\Gamma(b)\Gamma(1+c)\Gamma(1+d)\Gamma(a+b+c+d-1)}
{\Gamma(a+c)\Gamma(a+d)\Gamma(b+c)\Gamma(b+d)}.
 \enm
The combination of the last two equations produces
 \bnm
&&cd\sum_{k=0}^{\infty}\frac{\Gamma(c+k)\Gamma(d+k)\Gamma(a+b+c+d-1+k)}
{\Gamma(1+k)\Gamma(a+c+d+k)\Gamma(b+c+d+k)}\\[1mm]
&&+\:cd\sum_{k=0}^{\infty}\frac{\Gamma(a+k)\Gamma(b+k)\Gamma(a+b+c+d-1+k)}
{\Gamma(1+k)\Gamma(a+b+c+k)\Gamma(a+b+d+k)}\\[1mm]
&&=\frac{\Gamma(a)\Gamma(b)\Gamma(1+c)\Gamma(1+d)\Gamma(a+b+c+d-1)}
{\Gamma(a+c)\Gamma(a+d)\Gamma(b+c)\Gamma(b+d)}.
 \enm
Dividing both sides by $cd$, we achieve Theorem \ref{thm}.
\end{proof}

The symmetric formula is beautiful and it includes some known
results as special cases. On the research of reciprocity formulas,
the reader is referred to the papers
\cite{kn:andrews-a,kn:kang,kn:liu-a,kn:liu-b,kn:ma}.

When $b=-n$, Theorem \ref{thm} reduces to the following summation
formula:
 \bnm
\qquad _3F_2\ffnk{cccc}{1}{a,a+c+d-1-n,-n}{a+c-n,a+d-n}
=\frac{(1-c)_n(1-d)_n} {(1-a-c)_n(1-a-d)_n},
 \enm
which is equivalent to \eqref{saal}. Thus Theorem \ref{thm} can be
regarded as the nonterminating form of \eqref{saal}.

When $c=a$ and $d=b$, Theorem \ref{thm} reduces to the following
summation formula: \bnm
 _3F_2\ffnk{cccc}{1}{a,b,2a+2b-1}{a+2b,2a+b}
=\frac{1}{2}\frac{\Gamma(a)\Gamma(b)\Gamma(a+2b)\Gamma(2a+b)}
{\Gamma(2a)\Gamma(2b)\Gamma(a+b)\Gamma(a+b)}.
 \enm
It can also be attained by letting $a\to2a+2b-1$, $c\to a$ in
 Dixon's $_3F_2$-series identity(cf. \citu{andrews-b}{p. 72}):
 \bnm
 &&\xxqdn\xxqdn_3F_2\ffnk{cccc}{1}{a,b,c}{1+a-b,1+a-c}
  \nnm\\[1mm]
&&\xxqdn\xxqdn\:\:=\:\frac{\Gamma(1+\frac{a}{2})\Gamma(1+a-b)\Gamma(1+a-c)\Gamma(1+\frac{a}{2}-b-c)}
{\Gamma(1+a)\Gamma(1+\frac{a}{2}-b)\Gamma(1+\frac{a}{2}-c)\Gamma(1+a-b-c)},
 \enm
provided $Re(1+\frac{a}{2}-b-c)>0$.

Now we begin to prove \eqref{saal-non} by using Theorem \ref{thm}.
In terms of \eqref{f21}, we have
 \bnm
&&\qdn\xqdn\Phi(c,d;a,b)\\[1mm]
&&\qdn\xqdn=\sum_{k=0}^{\infty}\frac{\Gamma(c+k)\Gamma(d+k)\Gamma(a+b+c+d-1+k)}
{\Gamma(1+k)\Gamma(a+c+d+k)\Gamma(b+c+d+k)}\\[1mm]
&&\qdn\xqdn=\sum_{k=0}^{\infty}\frac{\Gamma(d+k)\Gamma(a+b+c+d-1+k)}
{\Gamma(1+k)\Gamma(a+b+c+2d+k)}{_2F_1}\ffnk{ccc}{1}{a+d,b+d}{a+b+c+2d+k}\\[1mm]
&&\qdn\xqdn=\sum_{k=0}^{\infty}\sum_{j=0}^{\infty}
\frac{\Gamma(d+k)\Gamma(a+b+c+d-1+k)\Gamma(a+d+j)\Gamma(b+d+j)}
{\Gamma(1+k)\Gamma(1+j)\Gamma(a+b+c+2d+k+j)\Gamma(a+d)\Gamma(b+d)}\\[1mm]
&&\qdn\xqdn=\sum_{j=0}^{\infty}\frac{\Gamma(a+d+j)\Gamma(b+d+j)\Gamma(d)\Gamma(a+b+c+d-1)}
{\Gamma(1+j)\Gamma(a+b+c+2d+j)\Gamma(a+d)\Gamma(b+d)}{_2F_1}\ffnk{ccc}{1}{d,a+b+c+d-1}{a+b+c+2d+j}\\[1mm]
&&\qdn\xqdn=\sum_{j=0}^{\infty}\frac{\Gamma(d)\Gamma(a+b+c+d-1)\Gamma(a+d+j)\Gamma(b+d+j)}
{\Gamma(a+d)\Gamma(b+d)\Gamma(1+d+j)\Gamma(a+b+c+d+j)}\\[1mm]
&&\qdn\xqdn=\frac{1}{d(a+b+c+d-1)}{_3F_2}\ffnk{ccc}{1}{1,a+d,b+d}{1+d,a+b+c+d}.
 \enm
Substitute the relation into Theorem \ref{thm} to obtain
 \bnm
&&\frac{\Gamma(a)\Gamma(b)\Gamma(a+b+c+d-1)}
{\Gamma(a+b+c)\Gamma(a+b+d)}{_3F_2}\ffnk{ccc}{1}{a,b,a+b+c+d-1}{a+b+c,a+b+d}
\\[1mm]&&+\:\frac{1}{d(a+b+c+d-1)}{_3F_2}\ffnk{ccc}{1}{1,a+d,b+d}{1+d,a+b+c+d}\\[1mm]
&&=\frac{\Gamma(a)\Gamma(b)\Gamma(c)\Gamma(d)\Gamma(a+b+c+d-1)}
{\Gamma(a+c)\Gamma(a+d)\Gamma(b+c)\Gamma(b+d)}.
 \enm
Performing the replacements $c\to d-a-b$, $d\to c-a-b$ in the last
equation, we get \eqref{saal-non} to complete the proof.

In a word, we have derived the nonterminating form of
Saalsch\"{u}tz's theorem \eqref{saal-non} from Dougall's $_2H_2$
series identity \eqref{h22} via the series rearrangement method. By
reversing the process, it is not difficult to realize that we can
also deduce \eqref{h22} from \eqref{saal-non} through the series
rearrangement method. In this sense, \eqref{saal-non} and
\eqref{h22} are equivalent with each other.

$${}$$\\

\section{Bailey's $_6\psi_6$ series identity implies the nonterminating form\\of Jackson's $_8\phi_7$ summation
formula}

The case $e=a$ of \eqref{bailey} is
 \bmn\label{phi-65}
  {_6\phi_5}\ffnk{cccccccccc}{q;\frac{qa}{bcd}}
 {a,q\sqrt{a},-q\sqrt{a},b,c,d}
 {\sqrt{a},-\sqrt{a},qa/b,qa/c,qa/d}=
 \frac{(qa,qa/bc,qa/bd,qa/cd;q)_{\infty}}{(qa/b,qa/c,qa/d,qa/bcd;q)_{\infty}},
 \emn
where $|qa/bcd|<1$. It is well known that Jackson's $_8\phi_7$
summation formula (cf. \citu{gasper}{p. 43})
  \bmn\label{jackon}
  &&\xxqdn{_8\phi_7}\ffnk{ccccccc}
 {q;q}{a,q\sqrt{a},-q\sqrt{a},b,c,d,q^{1+n}a^2/bcd,q^{-n}}{\sqrt{a},-\sqrt{a},qa/b,qa/c,qa/d,q^{-n}bcd/a,q^{1+n}a}
 \nnm\\[1mm]&&\xxqdn\:\:=\:
\frac{(qa,qa/bc,qa/bd,qa/cd;q)_n}{(qa/b,qa/c,qa/d,qa/bcd;q)_n}
  \emn
can be derived from \eqref{phi-65} in accordance with the series
rearrangement method. Now we start to prove \eqref{jackon-non}
according to \eqref{bailey}, \eqref{phi-65} and \eqref{jackon}.
Split the bilateral series on the left hand side of \eqref{bailey}
into two parts to achieve
  \bmn\label{equation-a}
&&\qqdn\xxqdn
 \frac{(q,qa/cd,qa/ce,qa/cf,qa/de,qa/df,qa/ef,qa/cdef,qcdef/a;q)_{\infty}}
 {(qc,qd,qe,qf,qa/cde,qa/cdf,qa/cef,qa/def,qa^2/cdef;q)_{\infty}}
 \nnm\\[1mm]\nnm
&&\qqdn\xxqdn\:={_6\psi_6}\ffnk{cccccccccc}{q;\frac{qa^2}{cdef}}
 {q\sqrt{\frac{cdef}{a}},-q\sqrt{\frac{cdef}{a}},cde/a,cdf/a,cef/a,def/a}
 {\sqrt{\frac{cdef}{a}},-\sqrt{\frac{cdef}{a}},qf,qe,qd,qc}\\[1mm]
 &&\qqdn\xxqdn\:=\sum_{k=0}^{\infty}\frac{1-q^{2k}cdef/a}{1-cdef/a}
\frac{(cde/a,cdf/a,cef/a,def/a;q)_k}{(qf,qe,qd,qc;q)_k}\bigg(\frac{qa^2}{cdef}\bigg)^k
  \nnm\\[1mm]\nnm
 &&\qqdn\xxqdn\:
  +\sum_{k=-\infty}^{-1}\frac{1-q^{2k}cdef/a}{1-cdef/a}
\frac{(cde/a,cdf/a,cef/a,def/a;q)_k}{(qf,qe,qd,qc;q)_k}\bigg(\frac{qa^2}{cdef}\bigg)^k\\[1mm]
 &&\qqdn\xxqdn\:=\sum_{k=0}^{\infty}
\frac{1-q^{2k}cdef/a}{1-cdef/a}\frac{(cde/a,cdf/a,cef/a,def/a;q)_k}{(qf,qe,qd,qc;q)_k}
 \bigg(\frac{qa^2}{cdef}\bigg)^k
  \nnm\\[1mm]\nnm
&&\qqdn\xxqdn\:+\:\frac{qa^2}{cdef}\frac{(1-q^2a/cdef)(1-1/c)(1-1/d)(1-1/e)(1-1/f)}{(1-a/cdef)(1-qa/def)(1-qa/cef)(1-qa/cdf)(1-qa/cde)}\\[1mm]
&&\qqdn\xxqdn\:\times\:\sum_{k=0}^{\infty}\frac{1-q^{2+2k}a/cdef}{1-q^2a/cdef}
  \frac{(q/c,q/d,q/e,q/f;q)_k}{(q^2a/def,q^2a/cef,q^2a/cdf,q^2a/cde;q)_k}
\bigg(\frac{qa^2}{cdef}\bigg)^k.
  \emn
Denote the two sums on the right hand side by $\Omega(a,c,d,e,f)$
and $\Theta(a,c,d,e,f)$, respectively. Above all, we can calculate
$\Omega(a,c,d,e,f)$ as follows:
 \bnm
&&\xxqdn\Omega(a,c,d,e,f)\\[1mm]
&&\xxqdn\:=\sum_{k=0}^{\infty}\frac{1-q^{2k}cdef/a}{1-cdef/a}\frac{(cdef/a,cef/a,def/a,cd/a;q)_{k}}
{(q,qd,qc,qef;q)_{k}}\bigg(\frac{qa^2}{cdef}\bigg)^k\\[1mm]
&&\xqdn\:\times\:{_8\phi_7}\ffnk{ccccccc}
 {q;q}{ef,q\sqrt{ef},-q\sqrt{ef},e,f,qa/cd,q^{k}cdef/a,q^{-k}}{\sqrt{ef},-\sqrt{ef},qf,qe,cdef/a,q^{1-k}a/cd,q^{1+k}ef}
 \enm
 \bnm
&&\xqdn\:=\sum_{i=0}^{\infty}\sum_{k=i}^{\infty}\frac{1-q^{2i}ef}{1-ef}
 \frac{(ef,e,f,qa/cd,q^{k}cdef/a,q^{-k};q)_i}{(q,qf,qe,cdef/a,q^{1-k}a/cd,q^{1+k}ef;q)_i}q^i\\[1mm]
&&\qdn\:\times\:
\frac{1-q^{2k}cdef/a}{1-cdef/a}\frac{(cdef/a,cef/a,def/a,cd/a;q)_k}{(q,qd,qc,qef;q)_k}
\bigg(\frac{qa^2}{cdef}\bigg)^k\\[1mm]
 &&\xqdn\:=\sum_{i=0}^{\infty}\frac{1-q^{2i}ef}{1-ef}
\frac{(ef,e,f,qa/cd,cef/a,def/a;q)_i(qcdef/a;q)_{2i}}
{(q,qf,qe,cdef/a,qd,qc;q)_i(qef;q)_{2i}}
\bigg(\frac{qa}{ef}\bigg)^i\\[1mm]
  &&\qdn\:\times\:\:
  {_6\phi_5}\ffnk{ccccccc}{q;\frac{qa^2}{cdef}}{q^{2i}cdef/a,q^{1+i}\!\sqrt{cdef/a},-q^{1+i}\!\sqrt{cdef/a},q^icef/a,q^idef/a,cd/a}
 {q^{i}\!\sqrt{cdef/a},-q^{i}\!\sqrt{cdef/a},q^{1+i}d,q^{1+i}c,q^{1+2i}ef}\\[1mm]
&&\xqdn\:=\frac{(qa/c,qa/d,qa/ef,qcdef/a;q)_{\infty}}{(qc,qd,qef,qa^2/cdef;q)_{\infty}}
\sum_{i=0}^{\infty}\frac{1-q^{2i}ef}{1-ef}\frac{(ef,e,f,ef/a;q)_i}{(q,qf,qe,qa;q)_i}
\bigg(\frac{qa}{ef}\bigg)^i\\[1mm]
&&\qdn\:\times\:{_8\phi_7}\ffnk{ccccccc}
{q;q}{a,q\sqrt{a},-q\sqrt{a},qa^2/cdef,c,d,q^{i}ef,q^{-i}}
{\sqrt{a},-\sqrt{a},cdef/a,qa/c,qa/d,q^{1-i}a/ef,q^{1+i}a}\\[1mm]
&&\xqdn\:=
\frac{(qa/c,qa/d,qa/ef,qcdef/a;q)_{\infty}}{(qc,qd,qef,qa^2/cdef;q)_{\infty}}
 \sum_{j=0}^{\infty}\frac{1-q^{2j}a}{1-a}\frac{(a,qa^2/cdef,c,d,e,f;q)_j}{(q,cdef/a,qa/c,qa/d,qf,qe;q)_j}
\\[1mm]
 &&\qdn\:\times\: \frac{(qef;q)_{2j}}{(qa;q)_{2j}}q^j\:
{_6\phi_5}\ffnk{ccccccc}{q;\frac{qa}{ef}}{q^{2j}ef,q^{1+j}\!\sqrt{ef},-q^{1+j}\!\sqrt{ef},q^je,q^jf,ef/a}
 {q^{j}\!\sqrt{ef},-q^{j}\!\sqrt{ef},q^{1+j}f,q^{1+j}e,q^{1+2j}a}\\[1mm]
 &&\xqdn\:=\frac{(q,qa/c,qa/d,qa/e,qa/f,qcdef/a;q)_{\infty}}{(qa,qc,qd,qe,qf,qa^2/cdef;q)_{\infty}}\\[1mm]
&&\xqdn\:\times\:{_8\phi_7}\ffnk{ccccccc}
 {q;q}{a,q\sqrt{a},-q\sqrt{a},qa^2/cdef,c,d,e,f}{\sqrt{a},-\sqrt{a},cdef/a,qa/c,qa/d,qa/e,qa/f}.
\enm
 Similarly, $\Theta(a,c,d,e,f)$ can be manipulated as
 \bnm
&&\xxqdn\Theta(a,c,d,e,f)\\[1mm]
&&\xxqdn=\frac{qa^2}{cdef}\frac{(1-q^2a/cdef)(1-1/c)(1-1/d)(1-1/e)(1-1/f)}{(1-a/cdef)(1-qa/def)(1-qa/cef)(1-qa/cdf)(1-qa/cde)}\\[1mm]
&&\xxqdn\times\:
\frac{(q,q^3a/cdef,q^2a^2/c^2def,q^2a^2/cd^2ef,q^2a^2/cde^2f,q^2a^2/cdef^2;q)_{\infty}}
 {(q^2a/cde,q^2a/cdf,q^2a/cef,q^2a/def,qa^2/cdef,q^3a^3/c^2d^2e^2f^2;q)_{\infty}}\\[1mm]
&&\xxqdn\:\times\: {_8\phi_7}\ffnk{ccccccccccccccccc}{q;q}
{\frac{q^2a^3}{c^2d^2e^2f^2},q\sqrt{\frac{q^2a^3}{c^2d^2e^2f^2}},-q\sqrt{\frac{q^2a^3}{c^2d^2e^2f^2}},
 \frac{qa^2}{cdef},\frac{qa}{cde},\frac{qa}{cdf},\frac{qa}{cef},\frac{qa}{def}}
 {\sqrt{\frac{q^2a^3}{c^2d^2e^2f^2}},-\sqrt{\frac{q^2a^3}{c^2d^2e^2f^2}},\frac{q^2a}{cdef},
 \frac{q^2a^2}{cdef^2},\frac{q^2a^2}{cde^2f},\frac{q^2a^2}{cd^2ef},\frac{q^2a^2}{c^2def}}.
 \enm
Substituting the last two equations into \eqref{equation-a}, we
attain \eqref{jackon-non} after some regular simplifications.

In conclusion, we have deduced the nonterminating form of Jackson's
$_8\phi_7$ summation formula \eqref{jackon-non} from Bailey's
$_6\psi_6$ series identity \eqref{bailey} via the series
rearrangement method. By reversing the process, it is conventional
to understand that we can also derive \eqref{bailey} from
\eqref{jackon-non} through the series rearrangement method. In this
sense, \eqref{jackon-non} and \eqref{bailey} are equivalent with
each other. $${}$$\\

 \textbf{Acknowledgments}

 The work is supported by the National Natural Science Foundation of China (No. 11661032).


\end{document}